\documentclass[12pt]{amsart}
\usepackage{amsfonts}
\usepackage{ifthen}
\usepackage{amsthm}
\usepackage{amsmath}
\usepackage{graphicx}
\usepackage{amscd,amssymb,amsthm}
\usepackage{graphicx}
\usepackage{epstopdf}
\usepackage{hyperref}
\usepackage{cite}
\usepackage[T1]{fontenc}
\usepackage[utf8]{inputenc}
\usepackage[dvipsnames]{xcolor}
\usepackage{cleveref}

\newcounter{minutes}
\divide\time by 60
\newcounter{hours}
\multiply\time by 60 \addtocounter{minutes}{-\time}

\newtheorem{lemma}{Lemma}
\newtheorem{theorem}{Theorem}
\newtheorem{corollary}{Corollary}

\newcommand{\real}{\operatorname{Re}}


\title[{Hardy class of $q$-Bessel functions}]{On some geometric properties and Hardy class of $q$-Bessel functions}

\author[\.{I}. Akta\c{s}]{\.{I}brah\.{I}m Akta\c{s}}
\address{Department of Mathematics, Kam\.{I}l \"{O}zda\u{g} Science Faculty, Karamano\u{g}lu Mehmetbey University, Yunus Emre Campus, 70100, Karaman--Turkey}
\email{aktasibrahim38@gmail.com; ibrahimaktas@kmu.edu.tr}

\begin{document}

\def\thefootnote{}
\footnotetext{ \texttt{File:~\jobname .tex,
          printed: \number\year-\number\month-\number\day,
          \thehours.\ifnum\theminutes<10{0}\fi\theminutes}
} \makeatletter\def\thefootnote{\@arabic\c@footnote}\makeatother

\maketitle

\begin{abstract}
In this paper, we deal with some geometric properties including starlikeness and convexity of order $\alpha$ of Jackson's second and third $q$-Bessel functions which are natural extensions of classical Bessel function $J_{\nu}$. In additon, we determine some conditions on the parameters such that Jackson's second and third $q$-Bessel functions belong to the Hardy space and to the class of bounded analytic functions. 
\end{abstract}

\section{Introduction and Preliminaries}
Special functions appears in many branches of mathematics and applied sciences. One of the most important special functions is Bessel function of the first kind $J_{\nu}$. The Bessel function of the first kind $J_{\nu}$ is a particular solution of the following homogeneous differential equation:
\begin{equation*}
z^{2}w^{\prime\prime}(z)+zw^{\prime}(z)+\left(z^{2}-\nu^{2}\right)w(z)=0,
\end{equation*}
which is known as Bessel differential equation. Also, the function $J_{\nu}$ has the following power series representation:
\begin{equation*}
J_{\nu}(z)=\sum_{n\geq0}\frac{(-1)^{n}\left(\frac{z}{2}\right)^{2n+\nu}}{n!\Gamma(n+\nu+1)},
\end{equation*}
where $\Gamma(z)$ denotes Euler's gamma function. Comprehensive information about the Bessel function can be found in Watson's treatise \cite{Watson}. On the other hand, there are some $q$-analogues of the Bessel functions in the literature. At the beginning of the 19. century, with the help of the $q$-calculus, famous English mathematician Frank Hilton Jackson has defined some functions which are known as Jackson's $q$-Bessel functions. The Jackson's second and third $q$-Bessel functions are defined by (see \cite{ANNABY1,ismailBook,Jackson1905})
\begin{equation}\label{second}
J_{\nu}^{(2)}(z;q)=\frac{(q^{\nu+1};q)_{\infty}}{(q;q)_{\infty}}\sum_{n\geq0}\frac{(-1)^{n}\left(\frac{z}{2}\right)^{2n+\nu}}{(q;q)_{n}(q^{\nu+1};q)_{n}}q^{n(n+\nu)}
\end{equation}
and
\begin{equation}\label{third}
J_{\nu}^{(3)}(z;q)=\frac{(q^{\nu+1};q)_{\infty}}{(q;q)_{\infty}}\sum_{n\geq0}\frac{(-1)^{n}z^{2n+\nu}}{(q;q)_{n}(q^{\nu+1};q)_{n}}q^{\frac{1}{2}{n(n+1)}},
\end{equation} where $z\in\mathbb{C},\nu>-1,q\in(0,1)$ and $$(a;q)_0=1,	(a;q)_n=\prod_{k=1}^{n}\left(1-aq^{k-1}\right),	 (a,q)_{\infty}=\prod_{k\geq1}\left(1-aq^{k-1}\right).$$
Here we would like to say that Jackson's third $q$-Bessel function is also known as Hahn-Exton $q$-Bessel function due to their contributions to the theory of $q-$Bessel functions (see \cite{koelink,Jackson1905-2}). In addition, it is known from \cite[p. 21]{ANNABY1} that these $q$-analogues satisfy the following limit relations:
$$\lim_{q\to1^{-}}J_{\nu}^{(2)}((1-q)z;q)=J_{\nu}(z)\text{ and }\lim_{q\to1^{-}}J_{\nu}^{(3)}((1-q)z;q)=J_{\nu}(2z).$$
Also, it is important to mention here that the notations \eqref{second} and \eqref{third} are from Ismail (see \cite{ismail1,ismailBook}) and they are different from the Jackson's original notations. Results on the properties of Jackson's second and third $q$-Bessel functions may be found in \cite{ANNABY1,ANNABY2,Koornwinder,abreu,ismail1,ismail2,koelink,aktas2017,aktas2018,aktas2019,BDM} and the references therein, comprehensively. 

This paper is organized as follow: the rest of this section is devoted to some basic concepts and results needed for the proof of
our main results. In Section 2, we deal with some geometric properties including  starlikeness and convexity of order $\alpha$ of normalized $q$-Bessel functions. Also, we present some results regarding Hardy space of normalized $q$-Bessel functions.

Let $\mathbb{U}=\{z\in\mathbb{C}:\left|z\right|<1\}$ be the open unit disk and $\mathcal{H}$ be the set of all analytic functions on the open unit disk $\mathbb{U}$. We denote by $\mathcal{A}$ the class of analytic functions $f:\mathbb{U}\rightarrow\mathbb{C},$ normalized by
\begin{equation}\label{f(z)}
f(z)=z+\sum_{n\geq2}a_{n}z^n.
\end{equation}
 By $\mathcal{S}$ we mean the class of functions belonging to $\mathcal{A}$ which are univalent in $\mathbb{U}$. Also, for $0\leq\alpha<1$, $\mathcal{S}^{\star}(\alpha)$ and $\mathcal{C}(\alpha)$ denote the subclasses of $\mathcal{S}$ consisting of all functions in $ \mathcal{A} $ which are starlike of order $\alpha$ and convex of order $\alpha$ in the open unit disk $\mathbb{U}$, respectively. When $\alpha=0$, we denote the classes $\mathcal{S}^{\star}(\alpha)$ and $\mathcal{C}(\alpha)$ by $\mathcal{S}^{\star}$ and $\mathcal{C}$, respectively. The analytic characterizations of these subclasses can be given as follows:
\begin{equation}\label{Starlike}
\mathcal{S}^{\star}(\alpha)=\left\lbrace f: f\in\mathcal{A}\text{ and }\Re\left(\frac{zf^{\prime}(z)}{f(z)}\right)>\alpha\text{ for }z\in\mathbb{U}\right\rbrace
\end{equation} 
and
\begin{equation}\label{Convex}
\mathcal{C}(\alpha)=\left\lbrace f: f\in\mathcal{A}\text{ and }\Re\left(1+\frac{zf^{\prime\prime}(z)}{f^{\prime}(z)}\right)>\alpha\text{ for }z\in\mathbb{U}\right\rbrace,
\end{equation}
respectively. In \cite{B06}, for $\alpha<1$, the author introduced the classes:
\begin{equation}\label{P(alpha)}
\mathcal{P}(\alpha)=\{p\in\mathcal{H}:\exists\eta\in\mathbb{R}\text{ such that } p(0)=1, \real\left[e^{i\eta}(p(z)-\alpha)\right]>0, z\in\mathbb{U}\}
\end{equation}
and
\begin{equation}\label{R(alpha)}
\mathcal{R}(\alpha)=\{f\in\mathcal{A}:\exists\eta\in\mathbb{R}\text{ such that } \real\left[e^{i\eta}(p(z)-\alpha)\right]>0, z\in\mathbb{U}\}.
\end{equation}
When $\eta=0$, the classes $\mathcal{P}(\alpha)$ and $\mathcal{R}(\alpha)$ will be denoted by $\mathcal{P}_{0}(\alpha)$ and $\mathcal{R}_{0}(\alpha)$, respectively. Also, for $\alpha=0$ we denote $\mathcal{P}_{0}(\alpha)$ and $\mathcal{R}_{0}(\alpha)$ simply by $\mathcal{P}$ and $\mathcal{R}$, respectively. In addition, the Hadamard product (or convolution) of two power series
$$f_{1}(z)=z+\sum_{n\geq2}a_{n}z^{n}\text{ and }f_{2}(z)=z+\sum_{n\geq2}b_{n}z^{n},$$
is defined by
$$(f_{1}*f_{2})(z)=z+\sum_{n\geq2}a_{n}b_{n}z^{n}.$$
Let $\mathcal{H}^{p}$ $(0<p\leq\infty)$  denote the Hardy space of all analytic functions $f(z)$ in $\mathbb{U}$, and define the integral means $M_{p}(r,f)$ by 
\begin{equation}\label{Hardy1}
M_{p}(r,f)=\begin{cases}
\left(\frac{1}{2\pi}\int_{0}^{2\pi}\left|f(re^{i\theta})\right|^{p}d\theta\right)^{\frac{1}{p}},\hspace{0.5cm}\text{ if } 0<p<\infty\\
\sup_{0\leq\theta<2\pi}\left|f(re^{i\theta})\right|, \hspace{1.3cm}\text{ if } p=\infty.
\end{cases}
\end{equation}
An analytic function $f(z)$ in $\mathbb{U}$, is said to belong to the Hardy space $\mathcal{H}^{p}$ where $0<p\leq\infty,$ if the set $\{ M_{p}(r,f): r\in\left[0,1 \right) \}$ is bounded. It is important to remind here that  $\mathcal{H}^{p}$ is a Banach space with the norm defined by (see \cite[p. 23]{Duren}) $$\left|\left|f\right|\right|_{p}=\lim_{r\to1^{-}}M_{p}(r,f)$$ for $1\leq{p}\leq\infty$. On the other hand, we know that $\mathcal{H}^{\infty}$ is the class of bounded analytic functions in $\mathbb{U},$ while $\mathcal{H}^{2}$ is the class of power series $\sum{a_{n}z^{n}}$ such that $\sum\left|a_{n}\right|^{2}<\infty.$ In addition, it is known from \cite{Duren} that $\mathcal{H}^{q}$ is a subset of $\mathcal{H}^{p}$ for $0<p\leq q\leq\infty$. Also, two well-knonw results about the Hardy space $\mathcal{H}^{p}$ are the following (see \cite{Duren}):
\begin{equation}\label{Hardy2}
\real{f^{\prime}(z)}>0\Rightarrow\begin{cases}
f^{\prime}\in\mathcal{H}^{q},\hspace{1.3cm} \forall{q<1}\\f\in\mathcal{H}^{\frac{q}{1-q}},\hspace{1cm}\forall{q\in(0,1)}.
\end{cases}
\end{equation}

In the recent years, the authors in \cite{ONSS02,Silverman75} and \cite{Singh82} proved several interesting results involving univalence, starlikeness, convexity and close-to-convexity of functions $f\in\mathcal{A}$. Later on, many authors investigated geometric properties of some special functions such as Bessel, Struve, Lommel, Mittag-Leffler and Wright by using the above mentioned results. Furthermore, Eenigenburg and Keogh determined some conditions on the convex, starlike and close-to-convex functions to belong to the Hardy space $\mathcal{H}^{p}$ in \cite{EK70}. On the other hand, the authors in \cite{B06,JKS93,KS94,Ponnusamy,Yagmur15,YO14,Prajabat} studied the Hardy space of some special functions (like Hypergeometric, Bessel, Struve, Lommel and Mittag-Leffler) and analytic function families.

Motivated by the above studies, our main aim is to determine some conditions on the parameters such that Jackson's second and third $q-$Bessel functions are starlike of order $\alpha$ and convex of order $\alpha$, respectively. Also, we find some conditions for the hadamard products $h_{\nu}^{(2)}(z;q)*f(z)$ and $h_{\nu}^{(3)}(z;q)*f(z)$ to belong to $\mathcal{H}^{\infty}\cap\mathcal{R}$, where $h_{\nu}^{(k)}(z;q)$ are Jackson's normalized $q$-Bessel functions which are given by \eqref{second2} and \eqref{third2} for $k\in\{2,3\}$, and $f$ is an analytic function in $\mathcal{R}$.  Morever, we investigate the Hardy space of the mentioned $q-$Bessel functions.

The next results will be used in order to prove several theorems.
\begin{lemma}[Silverman, \cite{Silverman75}]\label{Lemma1}
	Let f is of the form \eqref{f(z)}. If
	\begin{equation}\label{Starlikeness ineq.}
	\sum_{n=2}^{\infty}\left(n-\alpha\right)\left|a_n\right|\leq1-\alpha,
	\end{equation}
	then the function $f(z)$ is in the class ${\mathcal{S}^{\star}(\alpha)}$.
\end{lemma}

\begin{lemma}[Silverman, \cite{Silverman75}]\label{Lemma2}
	Let f is of the form \eqref{f(z)}. If
	\begin{equation}\label{convexity ineq.}
	\sum_{n=2}^{\infty}n\left(n-\alpha\right)\left|a_n\right|\leq1-\alpha,
	\end{equation}
	then the function $f(z)$ is in the class ${\mathcal{C}(\alpha)}.$
\end{lemma}

\begin{lemma}[Eenigenburg and Keogh, \cite{EK70}]\label{Lemma3}
	Let $\alpha\in\left[0,1\right) $. If the function $ f\in\mathcal{C}(\alpha) $ is not of the form
	\begin{equation}\label{forms}
	\begin{cases} f(z)=k+lz\left(1-ze^{i\theta}\right)^{2\alpha-1},\hspace{0.6cm}\alpha\neq\frac{1}{2} \\ f(z)=k+l\log\left(1-ze^{i\theta}\right),\hspace{1cm}\alpha=\frac{1}{2}\end{cases}
	\end{equation} 
	for some $k,l\in\mathbb{C}$ and $\theta\in\mathbb{R},$ then the following statements hold:
	\begin{enumerate}
		\item[\textbf{\textit{a.}}] There exist $\delta=\delta(f)>0$ such that $f^{\prime}\in\mathcal{H}^{\delta+\frac{1}{2(1-\alpha)}}.$
		\item[\textbf{\textit{b.}}] If $\alpha\in\left[ 0,\frac{1}{2}\right),$ then there exist $\tau=\tau(f)>0$ such that $f\in\mathcal{H}^{\tau+\frac{1}{1-2\alpha}}.$
		\item[\textbf{\textit{c.}}] If $\alpha\geq\frac{1}{2},$ then $f\in\mathcal{H}^{\infty}.$
	\end{enumerate}	
\end{lemma}
\begin{lemma}[Stankiewich and Stankiewich, \cite{Stankiewich}]\label{Lemma4}
$\mathcal{P}_{0}(\alpha)*\mathcal{P}_{0}(\beta)\subset\mathcal{P}_{0}(\gamma)$, where $\gamma=1-2(1-\alpha)(1-\beta)$. The value of $\gamma$ is the best possible.
\end{lemma}
In addition to the above Lemmas, in proving our assertions we will use some inequalities and series sums. It is easy to see that the following inequalites
\begin{equation}\label{ine.1}
q^{(n-1)(n-1+\nu)}\leq{q^{(n-1)\nu}},
\end{equation} 
\begin{equation}\label{ine.2}
q^{\frac{1}{2}(n-1)n}\leq{q^{\frac{1}{2}(n-1)}},
\end{equation}
\begin{equation}\label{ine.3}
(q;q)_{n-1}=\prod_{k=1}^{n-1}(1-q^{k})>(1-q)^{n-1},
\end{equation}
and
\begin{equation}\label{ine.4}
(q^{\nu+1};q)_{n-1}=\prod_{k=1}^{n-1}(1-q^{\nu+k})>(1-q^{\nu})^{n-1}
\end{equation}
hold true for $n\geq2$, $q\in(0,1)$ and $\nu>-1$. Furthermore, it can be easily shown that the following series sums
\begin{equation}\label{sum1}
\sum_{n\geq2}r^{n-1}=\frac{r}{1-r}, 
\end{equation}
\begin{equation}\label{sum2}
\sum_{n\geq2}nr^{n-1}=\frac{r(2-r)}{(1-r)^{2}},
\end{equation}
\begin{equation}\label{sum3}
\sum_{n\geq2}n^{2}r^{n-1}=\frac{r(r^{2}-3r+4)}{(1-r)^{3}}
\end{equation}
and
\begin{equation}\label{sum4}
\sum_{n\geq2}\frac{r^{n}}{n}=\log{\frac{1}{1-r}}-r
\end{equation}
hold true for $\left|r\right|<1$.

\section{Main results}
\setcounter{equation}{0}
In this section we present our main results. Due to the functions defined by \eqref{second} and \eqref{third} do not belong to the class $\mathcal{A}$, we consider following normalized forms of the $q$-Bessel functions:
\begin{equation}\label{second2}
h_{\nu}^{(2)}(z;q)=2^{\nu}c_{\nu}(q)z^{1-\frac{\nu}{2}}J_{\nu}^{(2)}(\sqrt{z};q)=z+\sum_{n\geq2}\frac{(-1)^{n-1}q^{(n-1)(n-1+\nu)}}{4^{n-1}(q;q)_{n-1}(q^{\nu+1};q)_{n-1}}z^{n}
\end{equation}
and
\begin{equation}\label{third2}
h_{\nu}^{(3)}(z;q)=c_{\nu}(q)z^{1-\frac{\nu}{2}}J_{\nu}^{(3)}(\sqrt{z};q)=z+\sum_{n\geq2}\frac{(-1)^{n-1}q^{\frac{1}{2}(n-1)n}}{(q;q)_{n-1}(q^{\nu+1};q)_{n-1}}z^{n},
\end{equation}
where $c_{\nu}(q)=(q;q)_{\infty}\big/(q^{\nu+1};q)_{\infty}$. As a result, these functions are in the class $\mathcal{A}$. Now, we are ready to present our main results related to the some geometric properties and Hardy class of Jackson's second and third $q$-Bessel functions.

\begin{theorem}\label{The1}
	Let $\alpha\in\left[0,1\right)$, $q\in(0,1)$, $\nu>-1$ and
	\begin{equation}\label{The1-2.3}
	4(1-q)(1-q^{\nu})-q^{\nu}>0.
	\end{equation} The following assertions are true:
	\begin{itemize}
		\item [\textbf{a.}] If the inequality
		\begin{equation}\label{The1-2.4}
		\alpha\geq\frac{q^{2\nu}+8(1-q)^{2}(1-q^{\nu})^{2}-8q^{\nu}(1-q)(1-q^{\nu})}{q^{2\nu}+8(1-q)^{2}(1-q^{\nu})^{2}-6q^{\nu}(1-q)(1-q^{\nu})}
		\end{equation}
		holds, then the normalized $q-$Bessel function $z\mapsto{h_{\nu}^{(2)}(z;q)}$ is starlike of order $\alpha$ in $\mathbb{U}$.
		\item [\textbf{b.}] If the inequality
		\begin{equation}\label{The1-2.5}
		\alpha\geq\frac{2q^{3\nu}-24q^{2\nu}(1-q)(1-q^{\nu})+112q^{\nu}(1-q)^{2}(1-q^{\nu})^{2}-64(1-q)^{3}(1-q^{\nu})^{3}}{2q^{3\nu}-24q^{2\nu}(1-q)(1-q^{\nu})+80q^{\nu}(1-q)^{2}(1-q^{\nu})^{2}-64(1-q)^{3}(1-q^{\nu})^{3}}
		\end{equation}
		holds, then the normalized $q-$Bessel function $z\mapsto{h_{\nu}^{(2)}(z;q)}$ is convex of order $\alpha$ in $\mathbb{U}$.
	\end{itemize}
	
\end{theorem}

\begin{proof}
	\begin{itemize}
		\item [\textbf{a.}]
	By virtue of the Silverman's result which is given in Lemma\eqref{Lemma1}, in order to prove the starlikeness of order $\alpha$ of the function, $z\mapsto{h_{\nu}^{(2)}(z;q)}$ it is enough to show that the following inequality
	\begin{equation}\label{kappa1}
	\kappa_{1}=\sum_{n\geq2}(n-\alpha)\left|\frac{(-1)^{n-1}q^{(n-1)(n-1+\nu)}}{4^{n-1}(q;q)_{n-1}(q^{\nu+1};q)_{n-1}}\right|\leq1-\alpha
	\end{equation}
	holds true under the hypothesis. Considering the inequalities \eqref{ine.1}, \eqref{ine.3} and \eqref{ine.4} together with the sums \eqref{sum1} and \eqref{sum2}, we may write that
	\begin{align*}
	\kappa_{1}&=\sum_{n\geq2}(n-\alpha)\left|\frac{(-1)^{n-1}q^{(n-1)(n-1+\nu)}}{4^{n-1}(q;q)_{n-1}(q^{\nu+1};q)_{n-1}}\right|\\&=\sum_{n\geq2}(n-\alpha)\frac{q^{(n-1)(n-1+\nu)}}{4^{n-1}(q;q)_{n-1}(q^{\nu+1};q)_{n-1}}\\&\leq\sum_{n\geq2}(n-\alpha)\frac{q^{(n-1)\nu}}{4^{n-1}(1-q)^{n-1}(1-q^{\nu})^{n-1}}\\&=\sum_{n\geq2}(n-\alpha)\left[\frac{q^{\nu}}{4(1-q)(1-q^{\nu})}\right]^{n-1}\\&=\sum_{n\geq2}n\left[\frac{q^{\nu}}{4(1-q)(1-q^{\nu})}\right]^{n-1}-\alpha\sum_{n\geq2}\left[\frac{q^{\nu}}{4(1-q)(1-q^{\nu})}\right]^{n-1}\\&=\frac{q^{\nu}\left(8(1-q)(1-q^{\nu})-q^{\nu}\right)}{\left(4(1-q)(1-q^{\nu})-q^{\nu}\right)^{2}}-\alpha\frac{q^{\nu}}{4(1-q)(1-q^{\nu})-q^{\nu}}.
	\end{align*}
	The inequality\eqref{The1-2.4} implies that the last sum is bounded above by $1-\alpha$. As a result, $\kappa_{1}\leq1-\alpha$ and so the function $z\mapsto{h_{\nu}^{(2)}(z;q)}$ is starlike of order $\alpha$ in $\mathbb{U}$.
	\item [\textbf{b.}] It is known from the Lemma\eqref{Lemma2} that to prove the convexity of order $\alpha$ of the function $z\mapsto{h_{\nu}^{(2)}(z;q)}$, it is enough to show that the following inequality
	\begin{equation}\label{kappa2}
	\kappa_{2}=\sum_{n\geq2}n(n-\alpha)\left|\frac{(-1)^{n-1}q^{(n-1)(n-1+\nu)}}{4^{n-1}(q;q)_{n-1}(q^{\nu+1};q)_{n-1}}\right|\leq1-\alpha
	\end{equation}
	is satisfied under our assumptions. Now, if we consider the inequalities \eqref{ine.1}, \eqref{ine.3} and \eqref{ine.4} together with the sums \eqref{sum2} and \eqref{sum3} then we may write that
	\begin{align*}
	\kappa_{2}&=\sum_{n\geq2}n(n-\alpha)\left|\frac{(-1)^{n-1}q^{(n-1)(n-1+\nu)}}{4^{n-1}(q;q)_{n-1}(q^{\nu+1};q)_{n-1}}\right|\\&=\sum_{n\geq2}n(n-\alpha)\frac{q^{(n-1)(n-1+\nu)}}{4^{n-1}(q;q)_{n-1}(q^{\nu+1};q)_{n-1}}\\&\leq\sum_{n\geq2}(n^{2}-n\alpha)\frac{q^{(n-1)\nu}}{4^{n-1}(1-q)^{n-1}(1-q^{\nu})^{n-1}}\\&=\sum_{n\geq2}(n^{2}-n\alpha)\left[\frac{q^{\nu}}{4(1-q)(1-q^{\nu})}\right]^{n-1}\\&=\sum_{n\geq2}n^{2}\left[\frac{q^{\nu}}{4(1-q)(1-q^{\nu})}\right]^{n-1}-\alpha\sum_{n\geq2}n\left[\frac{q^{\nu}}{4(1-q)(1-q^{\nu})}\right]^{n-1}\\&=\frac{q^{\nu}\left(64(1-q)^{2}(1-q^{\nu})^{2}-12q^{\nu}(1-q)(1-q^{\nu})+q^{2\nu}\right)}{\left(4(1-q)(1-q^{\nu})-q^{\nu}\right)^{3}}-\alpha\frac{q^{\nu}\left(8(1-q)(1-q^{\nu})-q^{\nu}\right)}{\left(4(1-q)(1-q^{\nu})-q^{\nu}\right)^{2}}.
	\end{align*}
	The inequality\eqref{The1-2.5} implies that the last sum is bounded above by $1-\alpha$. As a result, $\kappa_{2}\leq1-\alpha$ and so the function $z\mapsto{h_{\nu}^{(2)}(z;q)}$ is convex of order $\alpha$ in $\mathbb{U}$.
\end{itemize}
\end{proof}

\begin{theorem}\label{The2}
	Let $\alpha\in\left[0,1\right)$, $q\in(0,1)$, $\nu>-1$ and 
	\begin{equation}\label{The2-2.8}
	(1-q)(1-q^{\nu})-\sqrt{q}>0.
	\end{equation} The next two assertions are hold:
	\begin{itemize}
		\item [\textbf{a.}]If the inequality
		\begin{equation}\label{The2-2.9}
		\alpha\geq\frac{2\sqrt{q}(1-q)(1-q^{\nu})-q-\left((1-q)(1-q^{\nu})-\sqrt{q}\right)^{2}}{\left((1-q)(1-q^{\nu})-\sqrt{q}\right)\left(2\sqrt{q}-(1-q)(1-q^{\nu})\right)}
		\end{equation}
		holds, then the normalized $q-$Bessel function $z\mapsto{h_{\nu}^{(3)}(z;q)}$ is starlike of order $\alpha$ in $\mathbb{U}$.
		\item [\textbf{b.}]If the inequality
		\begin{equation}\label{The2-2.10}
		\alpha\geq\frac{4\sqrt{q}(1-q)^{2}(1-q^{\nu})^{2}-3q(1-q)(1-q^{\nu})+q\sqrt{q}-\left((1-q)(1-q^{\nu})-\sqrt{q}\right)^{3}}{\left((1-q)(1-q^{\nu})-\sqrt{q}\right)\left(2\sqrt{q}(1-q)(1-q^{\nu})-q-\left((1-q)(1-q^{\nu})-\sqrt{q}\right)^{2}\right)}
		\end{equation}
		holds, then the normalized $q-$Bessel function $z\mapsto{h_{\nu}^{(3)}(z;q)}$ is convex of order $\alpha$ in $\mathbb{U}$.
	\end{itemize}
	
\end{theorem}

\begin{proof}
	\begin{itemize}
		\item [\textbf{a.}]By virtue of the Silverman's result which is given in Lemma\eqref{Lemma1}, in order to prove the starlikeness of order $\alpha$ of the function, $z\mapsto{h_{\nu}^{(3)}(z;q)}$ it is enough to show that the following inequality
		\begin{equation}\label{kappa3}
		\kappa_{3}=\sum_{n\geq2}(n-\alpha)\left|\frac{(-1)^{n-1}q^{\frac{1}{2}n(n-1)}}{(q;q)_{n-1}(q^{\nu+1};q)_{n-1}}\right|\leq1-\alpha
		\end{equation}
		holds true under the hypothesis. Considering the inequalities \eqref{ine.2}, \eqref{ine.3} and \eqref{ine.4} together with the sums \eqref{sum1} and \eqref{sum2}, we may write that
		\begin{align*}
		\kappa_{3}&=\sum_{n\geq2}(n-\alpha)\left|\frac{(-1)^{n-1}q^{\frac{1}{2}n(n-1)}}{(q;q)_{n-1}(q^{\nu+1};q)_{n-1}}\right|\\&=\sum_{n\geq2}(n-\alpha)\frac{q^{\frac{1}{2}n(n-1)}}{(q;q)_{n-1}(q^{\nu+1};q)_{n-1}}\\&\leq\sum_{n\geq2}(n-\alpha)\frac{q^{\frac{1}{2}(n-1)}}{(1-q)^{n-1}(1-q^{\nu})^{n-1}}\\&=\sum_{n\geq2}n\left[\frac{\sqrt{q}}{(1-q)(1-q^{\nu})}\right]^{n-1}-\alpha\sum_{n\geq2}\left[\frac{\sqrt{q}}{(1-q)(1-q^{\nu})}\right]^{n-1}\\&=\frac{2\sqrt{q}(1-q)(1-q^{\nu})-q}{\left((1-q)(1-q^{\nu})-\sqrt{q}\right)^{2}}-\alpha\frac{\sqrt{q}}{(1-q)(1-q^{\nu})-\sqrt{q}}
		\end{align*}
		The inequality\eqref{The2-2.9} implies that the last sum is bounded above by $1-\alpha$. As a result, $\kappa_{3}\leq1-\alpha$ and so the function $z\mapsto{h_{\nu}^{(3)}(z;q)}$ is starlike of order $\alpha$ in $\mathbb{U}$.
		\item [\textbf{b.}]It is known from the Lemma\eqref{Lemma2} that to prove the convexity of order $\alpha$ of the function $z\mapsto{h_{\nu}^{(3)}(z;q)}$, it is enough to show that the following inequality
		\begin{equation}\label{kappa4}
		\kappa_{4}=\sum_{n\geq2}n(n-\alpha)\left|\frac{(-1)^{n-1}q^{\frac{1}{2}n(n-1)}}{(q;q)_{n-1}(q^{\nu+1};q)_{n-1}}\right|\leq1-\alpha
		\end{equation}
		is satisfied under the assumptions of Theorem\eqref{The4}. Now, if we consider the inequalities \eqref{ine.2}, \eqref{ine.3} and \eqref{ine.4} together with the sums \eqref{sum2} and \eqref{sum3} then we may write that
		\begin{align*}
		\kappa_{4}&=\sum_{n\geq2}n(n-\alpha)\left|\frac{(-1)^{n-1}q^{\frac{1}{2}n(n-1)}}{(q;q)_{n-1}(q^{\nu+1};q)_{n-1}}\right|\\&=\sum_{n\geq2}n(n-\alpha)\frac{q^{\frac{1}{2}n(n-1)}}{(q;q)_{n-1}(q^{\nu+1};q)_{n-1}}\\&\leq\sum_{n\geq2}(n^{2}-n\alpha)\frac{(\sqrt{q})^{n-1}}{(1-q)^{n-1}(1-q^{\nu})^{n-1}}\\&=\sum_{n\geq2}n^{2}\left[\frac{\sqrt{q}}{(1-q)(1-q^{\nu})}\right]^{n-1}-\alpha\sum_{n\geq2}n\left[\frac{\sqrt{q}}{(1-q)(1-q^{\nu})}\right]^{n-1}\\&=\frac{4\sqrt{q}(1-q)^{2}(1-q^{\nu})^{2}-3q(1-q)(1-q^{\nu})+q\sqrt{q}}{\left((1-q)(1-q^{\nu})-\sqrt{q}\right)^{3}}-\alpha\frac{2\sqrt{q}(1-q)(1-q^{\nu})-q}{\left((1-q)(1-q^{\nu})-\sqrt{q}\right)^{2}}.
		\end{align*}
		The inequality\eqref{The2-2.10} implies that the last sum is bounded above by $1-\alpha$. As a result, $\kappa_{4}\leq1-\alpha$ and so the function $z\mapsto{h_{\nu}^{(3)}(z;q)}$ is convex of order $\alpha$ in $\mathbb{U}$. 
	\end{itemize}
	 
\end{proof}

\begin{theorem}\label{The3} 
	Let $\alpha\in\left[0,1\right)$, $q\in(0,1)$ and $\nu>-1$. The following assertions hold true:
	\begin{itemize}
		\item [\textbf{a.}] If the inequlity\eqref{The1-2.3} is satisfied
		and
		\begin{equation}\label{The3-2.13}
		\alpha<\frac{4(1-q)(1-q^{\nu})-2q^{\nu}}{4(1-q)(1-q^{\nu})-q^{\nu}},
		\end{equation}
		then the function $\frac{h_{\nu}^{(2)}(z;q)}{z}$ is in the class $\mathcal{P}_{0}(\alpha)$.
		\item [\textbf{b.}] If the inequlity\eqref{The2-2.8} is satisfied
		and
		\begin{equation}\label{The3-2.14}
		\alpha<\frac{(1-q)(1-q^{\nu})-2\sqrt{q}}{(1-q)(1-q^{\nu})-\sqrt{q}},
		\end{equation}
		then the function $\frac{h_{\nu}^{(3)}(z;q)}{z}$ is in the class $\mathcal{P}_{0}(\alpha)$.
	\end{itemize}
	
\end{theorem}
\begin{proof}
\begin{itemize}
	\item [\textbf{a.}] In order to prove $\frac{h_{\nu}^{(2)}(z;q)}{z}\in\mathcal{P}_{0}(\alpha)$, it is enough to show that $\real\left(\frac{h_{\nu}^{(2)}(z;q)}{z}\right)>\alpha.$ For this purpose, consider the function $p(z)=\frac{1}{1-\alpha}\left(\frac{h_{\nu}^{(2)}(z;q)}{z}-\alpha\right)$. It can be easly seen that $\left|p(z)-1\right|<1$ implies $\real\left(\frac{h_{\nu}^{(2)}(z;q)}{z}\right)>\alpha.$ Now, using the inequalities \eqref{ine.1}, \eqref{ine.3}, \eqref{ine.4} and the well known geometric series sum, we have 
	\begin{align*}
	\left|p(z)-1\right|&=\left|\frac{1}{1-\alpha}\left[1+\sum_{n\geq2}\frac{(-1)^{n-1}q^{(n-1)(n-1+\nu)}}{4^{n-1}(q;q)_{n-1}(q^{\nu+1};q)_{n-1}}z^{n-1}-\alpha\right]-1\right|\\&\leq\frac{1}{1-\alpha}\sum_{n\geq2}\left(\frac{q^{\nu}}{4(1-q)(1-q^{\nu})}\right)^{n-1}\\&=\frac{q^{\nu}}{4(1-\alpha)(1-q)(1-q^{\nu})}\sum_{n\geq0}\left(\frac{q^{\nu}}{4(1-q)(1-q^{\nu})}\right)^{n}\\&=\frac{q^{\nu}}{(1-\alpha)\left[4(1-q)(1-q^{\nu})-q^{\nu}\right]}.
	\end{align*}
	It follows from the inequality \eqref{The3-2.13} that $\left|p(z)-1\right|<1,$ and hence $\frac{h_{\nu}^{(2)}(z;q)}{z}\in\mathcal{P}_{0}(\alpha)$.
	\item [\textbf{b.}] Similarly, let define the function $t(z)=\frac{1}{1-\alpha}\left(\frac{h_{\nu}^{(3)}(z;q)}{z}-\alpha\right)$. By making use of the inequalities \eqref{ine.2}, \eqref{ine.3}, \eqref{ine.4} and the geometric series sum, we can write that
	\begin{align*}
	\left|t(z)-1\right|&=\left|\frac{1}{1-\alpha}\left[1+\sum_{n\geq2}\frac{(-1)^{n-1}q^{\frac{1}{2}n(n-1)}}{(q;q)_{n-1}(q^{\nu+1};q)_{n-1}}z^{n-1}-\alpha\right]-1\right|\\&\leq\frac{1}{1-\alpha}\sum_{n\geq2}\left(\frac{\sqrt{q}}{(1-q)(1-q^{\nu})}\right)^{n-1}\\&=\frac{\sqrt{q}}{(1-\alpha)(1-q)(1-q^{\nu})}\sum_{n\geq0}\left(\frac{\sqrt{q}}{(1-q)(1-q^{\nu})}\right)^{n}\\&=\frac{\sqrt{q}}{(1-\alpha)\left[(1-q)(1-q^{\nu})-\sqrt{q}\right]}.
	\end{align*}
	But, the inequality \eqref{The3-2.14} implies that $\left|t(z)-1\right|<1.$ Therefore, we get $\frac{h_{\nu}^{(3)}(z;q)}{z}$ is in the class $\mathcal{P}_{0}(\alpha)$, and the proof is completed.
\end{itemize}
\end{proof}

Setting $\alpha=0$ and $\alpha=\frac{1}{2}$ in the Theorem\eqref{The3}, respectively, we have:
\begin{corollary}\label{Cor}
		Let $q\in(0,1)$ and $\nu>-1$. The next claims are true:
		\begin{itemize}
			\item [\textbf{i.}] If $2(1-q)(1-q^{\nu})-q^{\nu}>0$, then $\frac{h_{\nu}^{(2)}(z;q)}{z}\in\mathcal{P}$.
			\item [\textbf{ii.}] If $(1-q)(1-q^{\nu})-2\sqrt{q}>0$, then $\frac{h_{\nu}^{(3)}(z;q)}{z}\in\mathcal{P}$.
			\item [\textbf{iii.}] If $4(1-q)(1-q^{\nu})-3q^{\nu}>0$, then $\frac{h_{\nu}^{(2)}(z;q)}{z}\in\mathcal{P}_{0}\left(\frac{1}{2}\right)$.
			\item [\textbf{vi.}] If $(1-q)(1-q^{\nu})-3\sqrt{q}>0$, then $\frac{h_{\nu}^{(3)}(z;q)}{z}\in\mathcal{P}_{0}\left(\frac{1}{2}\right)$.
		\end{itemize}
\end{corollary}

\begin{theorem}\label{The4}
	Let $\alpha\in\left[0,1\right)$, $q\in(0,1)$, $\nu>-1$. If the inequalities \eqref{The1-2.3} and \eqref{The1-2.5} are satisfied, then 
	$h_{\nu}^{(2)}(z;q)\in\mathcal{H}^{\frac{1}{1-2\alpha}}$ for $\alpha\in\left[0,\frac{1}{2}\right)$ and $h_{\nu}^{(2)}(z;q)\in\mathcal{H}^{\infty}$ for $\alpha\in\left(\frac{1}{2},1\right).$
		
\end{theorem}
\begin{proof}
	It is known that Gauss hypergeometric function is defined by
\begin{equation}\label{Gauss}
_{2}F_{1}(a,b,c;z)=\sum_{n\geq0}\frac{(a)_{n}(b)_{n}}{(c)_{n}}\frac{z^n}{n!}.
\end{equation}
Now, using the equality\eqref{Gauss} it is possible to show that the function $z\mapsto h_{\nu}^{(2)}(z;q)$ can not be written in the forms which are given by \eqref{forms} for corresponding values of $\alpha$. More precisely, we can write that the following equalities: 
\begin{equation}\label{form1}
k+\frac{lz}{(1-ze^{i\theta})^{1-2\alpha}}=k+l\sum_{n\geq0}\frac{(1-2\alpha)_{n}}{n!}e^{i\theta{n}}z^{n+1}
\end{equation}
and
\begin{equation}\label{form2}
k+l\log{(1-ze^{i\theta})}=k-l\sum_{n\geq0}\frac{1}{n+1}e^{i\theta{n}}z^{n+1}
\end{equation}
hold true for $ k,l\in\mathbb{C}$ and $\theta\in\mathbb{R}$. If we consider the series representation of the function $z\mapsto h_{\nu}^{(2)}(z;q)$ which is given by \eqref{second2}, then we see that the function $z\mapsto h_{\nu}^{(2)}(z;q)$ is not of the forms \eqref{form1} for $\alpha\neq\frac{1}{2}$ and \eqref{form2} for $\alpha=\frac{1}{2}$, respectively. On the other hand, the part \textbf{b.} of Theorem\eqref{The1} states that the function $z\mapsto h_{\nu}^{(2)}(z;q)$ is convex of order $\alpha$ under hypothesis. Therefore, the proof is completed by applying Lemma\eqref{Lemma3}.
\end{proof}

\begin{theorem}\label{The5}
	Let $\alpha\in\left[0,1\right)$, $q\in(0,1)$, $\nu>-1$. If the inequalities \eqref{The2-2.8} and \eqref{The2-2.10} are satisfied, then $h_{\nu}^{(3)}(z;q)\in\mathcal{H}^{\frac{1}{1-2\alpha}}$ for $\alpha\in\left[0,\frac{1}{2}\right)$ and $h_{\nu}^{(3)}(z;q)\in\mathcal{H}^{\infty}$ for $\alpha\in\left(\frac{1}{2},1\right).$
\end{theorem}
\begin{proof}
From the power series representation of the function $z\mapsto{h_{\nu}^{(3)}(z;q)}$ which is given by \eqref{third2} it can be easily seen that this function is not of the forms \eqref{form1} for $\alpha\neq\frac{1}{2}$ and \eqref{form2} for $\alpha=\frac{1}{2}$, respectively. Also, we know from the second part of Theorem\eqref{The2} that the function $z\mapsto h_{\nu}^{(3)}(z;q)$ is convex of order $\alpha$ under our asumptions. As consequences, by applying Lemma\eqref{Lemma3} we have the desired results.
\end{proof}

\begin{theorem}\label{The6}
Let $q\in(0,1)$, $\nu>-1$ and $f(z)\in\mathcal{R}$ be of the form \eqref{f(z)}. The following statements are hold:
\begin{itemize}
	\item [\textbf{a.}] If $4(1-q)(1-q^{\nu})-3q^{\nu}>0$, then the hadamard product  $u(z)=h_{\nu}^{(2)}(z;q)*f(z)\in\mathcal{H}^{\infty}\cap\mathcal{R}.$
	\item [\textbf{b.}] If $(1-q)(1-q^{\nu})-3\sqrt{q}>0$, then the hadamard product  $v(z)=h_{\nu}^{(3)}(z;q)*f(z)\in\mathcal{H}^{\infty}\cap\mathcal{R}.$
\end{itemize}

\end{theorem}
\begin{proof}
	\begin{itemize}
		\item [\textbf{a.}] Suppose that the function $f(z)$ is in $\mathcal{R}.$ Then, from the definition of the class $\mathcal{R}$ we can say that the function $f^{\prime}(z)$ is in $\mathcal{P}.$ On the other hand, from the equality $u(z)=h_{\nu}^{(2)}(z;q)*f(z)$ we can easily see that $u^{\prime}(z)=\frac{h_{\nu}^{(2)}(z;q)}{z}*f^{\prime}(z).$ It is known from part \textbf{iii.} of the Corollary\eqref{Cor} that the function $\frac{h_{\nu}^{(2)}(z;q)}{z}\in\mathcal{P}_{0}\left(\frac{1}{2}\right).$ So, it follows from Lemma\eqref{Lemma4} that $u^{\prime}(z)\in\mathcal{P}.$ This means that $u(z)\in\mathcal{R}$ and $\real\left(u^{\prime}(z)\right)>0.$ If we consider the result which is given by \eqref{Hardy2}, then we have
		$u^{\prime}(z)\in\mathcal{H}^{p}$ for $p<1$ and $u(z)\in\mathcal{H}^{\frac{q}{1-q}}$ for $ 0<q<1 $, or equivalently, $u(z)\in\mathcal{H}^{p}$ for all $0<p<\infty.$ 
		
		Now, from the known upper bound for the \textit{Caratreodory functions} (see \cite[Theorem 1, p. 533]{MacGregor}), we have that, if the function $f(z)\in\mathcal{R}$, then $n\left|a_{n}\right|\leq2$ for $n\geq2.$ Using this fact together with the inequalities \eqref{ine.1}, \eqref{ine.3}, \eqref{ine.4} and the sum \eqref{sum4} we get
		\begin{align*}
		\left|u(z)\right|&=\left|h_{\nu}^{(2)}(z;q)*f(z)\right|=\left|z+\sum_{n\geq2}\frac{(-1)^{n-1}q^{(n-1)(n-1+\nu)}}{4^{n-1}(q;q)_{n-1}(q^{\nu+1};q)_{n-1}}a_{n}z^{n}\right|\\&\leq1+\frac{2}{\rho_{1}}\sum_{n\geq2}\frac{\rho_{1}^{n}}{n}=1+\frac{2}{\rho_{1}}\left(\log\frac{1}{1-\rho_{1}}-\rho_{1}\right),
		\end{align*}
		where $\rho_{1}=\frac{q^{\nu}}{4(1-q)(1-q^{\nu})}.$ This means that the function $z\mapsto{u(z)}$ is convergent absolutely for $\left|z\right|=1$ under the hypothesis. On the other hand, we know from \cite[Teorem 3.11, p.42]{Duren} that $u^{\prime}(z)\in\mathcal{H}^{q}$ implies the function $u(z)$ is continuous in $\overline{\mathbb{U}},$ where $\overline{\mathbb{U}}$ is closure of $\mathbb{U}.$  Since $\overline{\mathbb{U}}$ is a compact set, $u(z)$ is bounded in $\mathbb{U}$, that is, $u(z)\in\mathcal{H}^{\infty}$. 	
			 
		\item [\textbf{b.}] If $f(z)\in\mathcal{R},$ then $f^{\prime}(z)\in\mathcal{P}.$ Also, $v(z)=h_{\nu}^{(3)}(z;q)*f(z)$ implies $v^{\prime}(z)=\frac{h_{\nu}^{(3)}(z;q)}{z}*f^{\prime}(z).$ We known from part \textbf{vi.} of the Corollary\eqref{Cor} that the function $\frac{h_{\nu}^{(3)}(z;q)}{z}\in\mathcal{P}_{0}\left(\frac{1}{2}\right).$ So, by applying Lemma\eqref{Lemma4} we get $v^{\prime}(z)\in\mathcal{P}.$ That is, $v(z)\in\mathcal{R}$ and $\real\left(v^{\prime}(z)\right)>0.$ Now, from \eqref{Hardy2}, we have that
		$v^{\prime}(z)\in\mathcal{H}^{p}$ for $p<1$ and $v(z)\in\mathcal{H}^{\frac{q}{1-q}}$ for $ 0<q<1 $, or equivalently, $v(z)\in\mathcal{H}^{p}$ for all $0<p<\infty.$ 
		Using the well-known upper bound for the \textit{Caratreodory functions} together with the inequalities \eqref{ine.2}, \eqref{ine.3}, \eqref{ine.4} and the sum \eqref{sum4} we have
		\begin{align*}
		\left|v(z)\right|&=\left|h_{\nu}^{(3)}(z;q)*f(z)\right|=\left|z+\sum_{n\geq2}\frac{(-1)^{n-1}q^{\frac{1}{2}n(n-1)}}{(q;q)_{n-1}(q^{\nu+1};q)_{n-1}}a_{n}z^{n}\right|\\&\leq1+\frac{2}{\rho_{2}}\sum_{n\geq2}\frac{\rho_{2}^{n}}{n}=1+\frac{2}{\rho_{2}}\left(\log\frac{1}{1-\rho_{2}}-\rho_{2}\right),
		\end{align*}
		where $\rho_{2}=\frac{\sqrt{q}}{(1-q)(1-q^{\nu})}.$ So, we can say that the function $z\mapsto{v(z)}$ is convergent absolutely for $\left|z\right|=1$ under the stated conditions. Also, it is known from \cite[Teorem 3.11, p.42]{Duren} that $v^{\prime}(z)\in\mathcal{H}^{q}$ implies the function $v(z)$ is continuous in $\overline{\mathbb{U}}$. Since $\overline{\mathbb{U}}$ is a compact set, we may write that $v(z)$ is bounded in $\mathbb{U}$. Hence $v(z)\in\mathcal{H}^{\infty}$ and the proof is completed. 		
	\end{itemize}
\end{proof}
\begin{theorem}\label{The7}
	Let $\alpha\in\left[0,1\right)$, $\beta<1$, $\gamma=1-2(1-\alpha)(1-\beta)$, $q\in(0,1)$ and $\nu>-1$. Suppose that the function $f(z)$ of the form \eqref{f(z)} is in the class $\mathcal{R}_{0}(\beta).$ The following statements hold true:
	\begin{itemize}
		\item [\textbf{a.}] If the inequalities \eqref{The1-2.3} and \eqref{The3-2.13} are hold,
		then $u(z)=h_{\nu}^{(2)}(z;q)*f(z)\in\mathcal{R}_{0}(\gamma).$
		
		\item [\textbf{b.}] If the inequalities \eqref{The2-2.8} and \eqref{The3-2.14} are hold,
		then $v(z)=h_{\nu}^{(3)}(z;q)*f(z)\in\mathcal{R}_{0}(\gamma).$
	\end{itemize} 
\end{theorem}

\begin{proof}
	\begin{itemize}
		\item [\textbf{a.}] If $f(z)\in\mathcal{R}_{0}(\beta)$, then this implies that $f^{\prime}(z)\in\mathcal{P}_{0}(\beta).$ We know from the first part of Theorem\eqref{The3} that the function $\frac{h_{\nu}^{(2)}(z;q)}{z}$ is in the class $\mathcal{P}_{0}(\alpha).$ Since $u^{\prime}(z)=\frac{h_{\nu}^{(2)}(z;q)}{z}*f^{\prime}(z),$ taking into acount the Lemma\eqref{Lemma4} we may write that $u^{\prime}(z)\in\mathcal{P}_{0}(\gamma).$ This implies that $u(z)\in\mathcal{R}_{0}(\gamma).$ 
		\item [\textbf{b.}] Similarly, $f(z)\in\mathcal{R}_{0}(\beta)$ implies that $f^{\prime}(z)\in\mathcal{P}_{0}(\beta).$ It is known from the second part of Theorem\eqref{The3} that the function $\frac{h_{\nu}^{(3)}(z;q)}{z}$ is in the class $\mathcal{P}_{0}(\alpha).$ Using the fact that $v^{\prime}(z)=\frac{h_{\nu}^{(3)}(z;q)}{z}*f^{\prime}(z)$ and Lemma\eqref{Lemma4}, we have  $v^{\prime}(z)\in\mathcal{P}_{0}(\gamma).$ As a result, $v(z)\in\mathcal{R}_{0}(\gamma).$ 
	\end{itemize}
\end{proof}


\begin{thebibliography}{99}
\bibitem{abreu} \textsc{L. D. Abreu},
\textit{A q-sampling theorem related to the q-Hankel transform}, Proc. Amer. Math. Soc., \textbf{133}(4) (2005), 1197--1203.

	
\bibitem{aktas2017} \textsc{\.{I}. Akta\c{s}, \'A. Baricz}, \textit{Bounds for the radii of starlikeness of some $q$-Bessel functions}. Results Math., \textbf{72}(1--2) (2017), 947--963.


\bibitem{aktas2018} 
\textsc{\.{I}. Akta\c{s}, H. Orhan}, {\em On Partial sums of Normalized $ q $-Bessel Functions}, Commun. Korean Math. Soc., \textbf{33}(2) (2018), 535--547.

\bibitem{aktas2019}
\textsc{\.{I}. Akta\c{s}, H. Orhan}, \textit{Bounds for the radii of convexity of some $q$-Bessel functions}, Bull. Korean Math. Soc., (accepted).

\bibitem{ANNABY1} \textsc{M. H. Annaby and Z. S. Mansour}, \textit{$q$-Fractional Calculus and Equations (Lecture Notes in Mathematics 2056)}, Springer-Verlag, Berlin, 2012.

\bibitem{ANNABY2} \textsc{M. H. Annaby, Z. S. Mansour, O. A. Ashour},
\textit{Sampling theorems associated with biorthogonal q-Bessel functions}
J. Phys. A, \textbf{43}(29) (2010), Art. No. 295204.

\bibitem{B06}\textsc{\'A. Baricz}, \textit{Bessel transforms and Hardy space of generalized Bessel functions}, Mathematica, \textbf{48} (2006), 127--136.

\bibitem{BDM}
\textsc{\'A. Baricz, D.K. Dimitrov, I. Mez\H{o}}, \textit{Radii of starlikeness and convexity of some $q$-Bessel functions}, J. Math. Anal. Appl., \textbf{435} (2016), 968--985.
	
\bibitem{Duren}\textsc{P. L. Duren}, \textit{Theory of $\mathcal{H}^{p}$ spaces}. A series of Monographs and Textbooks in Pure and Applied Mathematics, vol. 38, Academic Press, New York and London, 1970.

\bibitem{EK70}\textsc{P. J. Eenigenburg, F. R. Keogh}, \textit{The Hardy class of some univalent functions and their derivatives}. Michigan Math. J., \textbf{17} (1970), 335--346.

\bibitem{ismail1}
\textsc{M.E.H. Ismail}, \textit{The zeros of basic Bessel function, the functions $J_{\nu+\alpha x}(x)$, and associated orthogonal polynomials}, J. Math. Anal. Appl., \textbf{86} (1982), 1--19.

\bibitem{ismailBook} \textsc{M.E.H. Ismail}, \textit{Classical and Quantum Orthogonal Polynomials in One Variable}, Cambridge University Press, Cambridge, 2005.

\bibitem{ismail2}
\textsc{M.E.H. Ismail, M.E. Muldoon}, \textit{On the variation with respect to a parameter of zeros of Bessel and $q$-Bessel functions}, J. Math. Anal. Appl., \textbf{135}(1) (1988), 187--207.

\bibitem{Jackson1905} \textsc{F. H. Jackson}, \textit{The Basic Gamma-Function and the Elliptic Functions}, Proc. R. Soc. Lond. A, \textbf{76} (1905), 127-144.

\bibitem{Jackson1905-2} \textsc{F. H. Jackson}, \textit{The applications of basic numbers to Bessel’s and Legendre’s equations}, Proc. Lond. Math. Soc. \textbf{3}(2) (1905), 1–23.

\bibitem{JKS93} \textsc{I. B. Jung, Y. C. Kim, H. M. Srivastava}, \textit{The Hardy space of analytic functions associated with certain one-parameter families of integral operators} J. Math. Anal. Appl., \textbf{176} (1993), 138--147.

\bibitem{KS94} \textsc{Y. C. Kim, H. M. Srivastava} \textit{Some families of generalized Hypergeometric functions associated with the Hardy space of analytic functions}, Prod. Japan Acad., Seri A, \textbf{70}(2) (1994), 41--46.

\bibitem{koelink}
\textsc{H.T. Koelink, R.F. Swarttouw}, \textit{On the zeros of the Hahn-Exton $q$-Bessel function and associated $q$-Lommel polynomials}, J. Math. Anal. Appl., \textbf{186} (1994), 690--710.

\bibitem{Koornwinder}
\textsc{T.H. Koornwinder, R.F. Swarttouw}, \textit{On $q$-analogues of the Hankel and Fourier transforms}, Trans. Amer. Math. Soc., \textbf{333} (1992), 445--461.

\bibitem{MacGregor} \textsc{T.H. MacGregor}, \textit{Functions whose derivative has a positive real part}, Trans. Amer. Math. Soc., \textbf{104} (1962), 532--537.

\bibitem{ONSS02} \textsc{S. Owa, M. Nunokawa, H. Saitoh, H. M. Srivastava} \textit{Close-to-convexity, starlikeness, and convexity of certain analytic functions}, App. Math. Letter., \textbf{15} (2002), 63--69.

\bibitem{Ponnusamy}\textsc{S. Ponnusamy}, \textit{The Hardy space of hypergeometric functions}, Complex Variables and Elliptic Equations, \textbf{29} (1996), 83--96.

\bibitem{Prajabat}\textsc{J.K. Prajapat, S. Maharana, D Bansal}, \textit{Radius of Starlikeness and Hardy Space of Mittag-Leffer Functions}, Filomat, \textbf{32}(18) (2018), 6475--6486.

\bibitem{Silverman75}\textsc{H. Silverman}, \textit{Univalent functions with negative coefficients}. Proc. Am. Math. Soc., \textbf{51}(1) (1975), 109--116.	

\bibitem{Singh82} \textsc{R. Singh, S. Singh}, \textit{Some sufficient conditions for univalence and starlikeness}. Colloq. Math., \textbf{47} (1982), 309--314. 

\bibitem{Stankiewich} \textsc{J. Stankiewich, Z. Stankiewich}, \textit{Some applications of Hadamard convolutions in the theory of functions}, Ann. Univ. Mariae Curie-Sklodowska, \textbf{40} (1986), 251--265.

\bibitem{Watson} \textsc{G. N. Watson}, A treatise on the theory of Bessel functions, Cambridge University Press, Cambridge,1944.

\bibitem{Yagmur15}\textsc{N. Yagmur}, \textit{Hardy space of Lommel functions}. Bull. Korean Math. Soc., \textbf{52}(3) (2015), 1035--1046.

\bibitem{YO14} \textsc{N. Yagmur, H. Orhan}, \textit{Hardy space of generalized Struve functions}, Complex Variables and Elliptic Equations, \textbf{59}(7) (2014), 929–936.


\end{thebibliography}
\end{document}